\documentclass[11pt,a4paper]{article}
\usepackage{a4wide,amsmath,amsthm,tikz}  
\usepackage{amsfonts,amssymb}
\usepackage{color}
\usepackage{thm-restate}
\usepackage{comment}

\theoremstyle{plain}
\newtheorem{theorem}{Theorem}
\newtheorem{lemma}[theorem]{Lemma}
\newtheorem{claim}[theorem]{Claim}
\newtheorem{question}[theorem]{Question}

\def\Evec{\vec E}
\def\vecE{\Evec}
\def\Gvec{\vec G}

\def\calI{\mathcal{I}}

\def\nn{\sqrt{n}}
\def\deg{\mathrm{deg}}

\newcommand{\Exp}{\mathop{\null\mathbb{E}}}
\let\rho=\varrho 
\newcommand{\Ptilde}{\widetilde P}
\let\epsilon=\varepsilon
\let\eps=\varepsilon
\renewcommand{\Pr}{\mathbb{P}}

\title{On the self-intersection time of non-backtracking random walks}
\author{Ferenc Bencs\thanks{Centrum Wiskunde \& Informatica, Amsterdam. \texttt{ferenc.bencs@cwi.nl}}, Leslie Ann Goldberg\thanks{Department of Computer Science, University of Oxford.}, Matthew Jenssen\thanks{King's College London, Department of Mathematics. \texttt{matthew.jenssen@kcl.ac.uk} }, Mark Jerrum\thanks{School of Mathematical Sciences, Queen Mary, University of London, UK\@. \texttt{m.jerrum@qmul.ac.uk}},  G\'abor Pete\thanks{HUN-REN Alfréd Rényi Institute of Mathematics, and Department of Stochastics, Institute of Mathematics, Budapest University of Technology and Economics, Hungary. \texttt{gabor.pete@renyi.hu}},\\ Guus Regts\thanks{Korteweg de Vries Institute for Mathematics, University of Amsterdam. \texttt{g.regts@uva.nl}}, and Yitong Yin\thanks{School of Computer Science, Nanjing University, China. \texttt{yinyt@nju.edu.cn}}}
\date{10 Aug 2026}

\begin{document}

\maketitle

\begin{abstract}
We study the self-intersection time of the non-backtracking random walk on connected undirected graphs. For every fixed $\Delta \geq 3$ we show that the expected self-intersection time is $O(\sqrt{n} \log n)$ on $n$-vertex graphs with minimum degree at least~$3$ and maximum degree at most~$\Delta$. For regular graphs with a uniform spectral gap, we improve this to $O(\sqrt{n})$. We also show an $\Omega(\sqrt{n})$ lower bound on a class of regular expanders. Our upper bound on the expected self-intersection time implies an improved mixing time bound on Glauber dynamics for the Ising model on $\Delta$-regular graphs at the tree uniqueness threshold.
\end{abstract}

\section{Introduction}

Let $G$ be a connected $n$-vertex graph with minimum degree $\delta\ge 3$ and maximum degree~$\Delta$.
A \emph{non-backtracking random walk} (NBW) on~$G$ 
from a vertex $u\in V(G)$ 
is a process $(X_t)_{t\geq 0}$ with $X_0=u$ defined as follows.
\begin{itemize}
\item At time $t=1$, $X_1$ is chosen uniformly from the set $N(X_0)$ of neighbours of~$X_0$.
\item For each $t\geq 2$, $X_t$ is chosen uniformly from $N(X_{t-1}) \setminus X_{t-2}$.
\end{itemize}
The \emph{self-intersection time} of the NBW is defined 
as \[T(G,u):= \min \{ t\geq 2 \colon 
\mbox{for some $j\in \{0,\ldots,t-1\}$, } 
X_t = X_j \}.\]

Alon, Benjamini, Lubetzky, and Sodin~\cite{ABLS} studied NBWs on regular expanders. (By convention, an ``expander'' is a vertex expander, and a family of expanders has expansion factor bounded uniformly away from 1.) 
The non-backtracking walk is not a Markov chain when viewed only on vertices, since the law of the next step depends on the previous vertex. However, as they observed, it is a Markov chain on directed edges. 
Indeed, suppose (for now) that $G$ is non-bipartite and $\Delta$-regular and let
$
\vec E:=\{(x,y):\{x,y\}\in E(G)\}
$
be the set of directed edges, so that
\(
N:=|\vec E|=n\Delta.
\) Let $\Gvec$ be the digraph with vertex set~$V(G)$ and edge set~$\vecE$.
For \(t\ge 1\), define
$
Y_t:=(X_{t-1},X_t)$.
Then \((Y_t)_{t\ge 1}\) is a Markov chain on \(\vec E\), with transition
matrix
\[
\vec P((x,y),(y,z))
=
\begin{cases}
\dfrac1{\Delta-1}, & \text{if }  z\neq x,\\[4pt]
0, & \text{otherwise}.
\end{cases}
\]

Since \(G\) is \(\Delta\)-regular, \(\vec P\) is doubly stochastic, hence its stationary
distribution is uniform on \(\vec E\). Moreover, since \(G\) is connected,
non-bipartite, and \(\Delta\ge 3\), this directed-edge chain converges to its stationary distribution. It follows that $X_t$ also converges to the uniform distribution.

Alon et al.~\cite{ABLS} define the mixing 
rate of $(X_t)$ as follows:
\begin{equation}\label{eq:intro-mixrate}
\tilde\rho(G)=\limsup_{t\to\infty}\max_{u,v\in V(G)}\Bigl|\Ptilde_{uv}^{(t)}-\frac1n\Bigr|^{1/t},
\end{equation}
where $\Ptilde^{(t)}_{uv}$ is the probability that an NBW of length~$t$ on~$G$, which starts at~$u$, ends at~$v$. Let 
$
\textrm{Spec}(A_G)
$
denote the set of eigenvalues of the adjacency matrix $A_G$ of $G$ and let 
\[
\lambda(G):=\max \{|\lambda|: \lambda\in \textrm{Spec}(A_G), |\lambda|\not=\Delta\}
\]
be the largest non-trivial eigenvalue in absolute value. 
They show that the NBW mixing rate is bounded away from 1 by a function of~$\lambda(G)$ and~$\Delta$. In particular, for an absolute expander, $X_t$ will be very close to being uniformly distributed in $O(\log n)$ steps. Combining this with the birthday paradox, which says that, when balls are thrown u.a.r.\null{} into $n$ bins, we expect a collision after $\Theta(\sqrt{n})$ balls, it is easy to conclude that $\Exp(T(G,u))=O(\sqrt{n}\log n)$.  
In this work, we obtain a similar upper bound for general bounded-degree graphs and a stronger, asymptotically tight bound for expanders.   

Our main theorem is Theorem~\ref{thm:UB}, which is an $ O(\sqrt{n}\log n)$ upper bound on the expected self-intersection time of non-backtracking walks in bounded-degree graphs.

\begin{restatable}{theorem}
{thmUB}\label{thm:UB}
Fix $\Delta \geq 3$ and let $G$ be a connected $n$-vertex graph with minimum degree at least~$3$ and maximum degree~$\Delta$. Fix $u\in V(G)$. Then $\Exp(T(G,u)) = O_\Delta(\sqrt{n} \log n)$.
\end{restatable}

For regular graphs with a uniform spectral gap,
Theorem~\ref{thm:UBexpander} improves the bound on the expected self-intersection time to $O(\sqrt{n})$.

\begin{restatable}{theorem}{thmUBexpander}
\label{thm:UBexpander}
Fix \(\Delta\ge 3\), \(\varepsilon>0\) and let \(G\) be a connected $n$-vertex, \(\Delta\)-regular graph with 
$
\lambda(G)\le \Delta
-\varepsilon$.
Fix $u\in V(G)$. Then 
$\Exp(T(G,u)) \leq O_{\Delta,\eps}(\sqrt{n})$.
\end{restatable}

How close are these results to being sharp? It was already observed in \cite[Remark 3.5]{ABLS} that if every vertex in a $\Delta$-regular $n$-vertex graph~$G$ is contained in a cycle of length at most $\epsilon \log_{\Delta-1} n$, then $\Exp(T(G,u)) = O(n^{\epsilon+o(1)})$. 
Thus, in order to obtain a lower bound that gets close to the above upper bounds which are $\widetilde{O}(\sqrt{n})$, one should consider high-girth graphs. Along these lines, our Theorem~\ref{thm:LB} gives an $\Omega(\sqrt{n})$ lower bound on the expected self-intersection time of NBWs on certain regular high-girth expanders. This shows that Theorem~\ref{thm:UB} is best possible up to a factor of $O(\log n)$, and that Theorem~\ref{thm:UBexpander} is sharp up to a constant factor.
 
\begin{restatable}{theorem}{thmLB}\label{thm:LB}
Let $\Delta=p+1$ where $p$ is a prime congruent to $1 \pmod4$.  Then there is an increasing sequence of regular graphs of degree~$\Delta$ such that, for each $n$-vertex graph $G$ in the sequence
and $u\in V(G)$, $\Exp(T(G,u)) \geq \sqrt{n}/6$.
\end{restatable}

In light of our results, we pose the following question.

\begin{question}\label{thequestion}
Does  $\Exp(T(G,u))=O(\sqrt{n})$ hold for every $n$-vertex  graph with minimum degree $\delta \ge 3$?
\end{question}

Note that the definition of a NBW can be extended naturally to multigraphs, and the answer to  
the analogous generalisation of
Question~\ref{thequestion} is ``no''. For example,   for the multigraph $M_n$ on $\{1,2,\dots,n\}$ with $2^{k^2}$ parallel edges between each
$k$ and $k+1$, the NBW starting at $X_0=1$ will, with a uniformly positive probability, have $X_k=k+1$ for all $k\leq n-1$, and hence $\Exp(T(M_n,1))=\Theta(n)$.
However, we do not know a similar construction for (simple) graphs.

Although NBWs and the birthday paradox for Markov chains have already been studied \cite{ABLS, kim2008birthday},  our questions apparently have not been previously addressed in the literature. 
In addition to their intrinsic interest, our main motivation for this work is an application to the mixing time of the Ising Glauber dynamics on $\Delta$-regular graphs at inverse temperature $\beta$ with absolute value at the tree uniqueness threshold: $|\beta|=\beta_c$, which is defined by $\tanh \beta_c = \frac{1}{\Delta-1}$. Following the work of~\cite{chen2020contraction} and~\cite{chen2025uniqueness}, our NBW result implies an $\widetilde{O}(\sqrt{n})$ bound on the spectral independence introduced in~\cite{anari2021spectral} (to our knowledge, the first sublinear bound for such a critical system), which in turn gives a bound of $\widetilde{O}\!\left(n^{2+ \frac{2}{\Delta - 2}}\right)$ on the mixing time of the discrete time Ising Glauber dynamics, improving on the previous best bound $\widetilde{O}\!\left(n^{3+\frac{4}{\Delta-2}}\right)$ from~\cite{chen2025uniqueness}\footnote{We note that the stronger bound $\widetilde{O}(n^{2+\frac{2}{\Delta-2}}$) was stated in an earlier conference version of \cite{chen2025uniqueness}, but the corresponding spectral-independence argument was later found to contain an error and was corrected in the arXiv version, which gives the weaker bound $\widetilde{O}\!\left(n^{3+\frac{4}{\Delta-2}}\right)$.}. 
See Theorem~\ref{thm:critical-Ising} in Section~\ref{sec:appli} for details.

\section{Upper bound on self-intersection time}
Throughout we fix an $n$-vertex graph $G$ of maximum degree $\Delta$ and minimum degree $\delta\ge 3$.  
We assume that $n$ is sufficiently large with respect to~$\Delta $ and we let   $d= \Delta-1$. 
 
Given vertices $u$ and $v$ that are adjacent in~$G$,
we define a rooted tree $T_{u,v}$ as follows.
Each node $\rho$ of $T_{u,v}$ has a label $\ell(\rho)$ 
which is a non-backtracking random walk (NBW) in $G$ starting with $uv$
and a mass $\pi(\rho)$ in the range $1/(d\nn) \leq \pi(\rho) \leq 1$. 
The root~$\rho^*$ has $\ell(\rho^*) = uv$ and $\pi(\rho^*) = 1$.
Here is how the tree grows further from a node~$\rho$. If $\pi(\rho) \leq 1/\nn$ then $\rho$ is a leaf. Otherwise, suppose that $\ell(\rho) = u_1 \cdots u_{m-1} u_m$. For each neighbour~$w\neq u_{m-1}$ of~$u_m$, create a child~$\rho_w$ of~$\rho$ with $\ell(\rho_w) = u_1
\cdots u_m w$ and $\pi(\rho_w) = \pi(\rho)/(\deg(u_m)-1)$.
By construction, each leaf~$\rho$ has
$1/(d\nn) \leq \pi(\rho) \leq 1/\nn$, and the sum of the masses of the leaves is~$1$. Also, $\pi(\rho)$ is the probability that a random NBW starting from $uv$ starts as $\ell(\rho)$, i.e., has $\ell(\rho)$ as a prefix.

Let $\calI$ be the set of NBWs that self-intersect. 
Given an NBW $P$, let $P_0$ be the 
beginning vertex in~$P$ and let $P_1$ be the next vertex in~$P$. Let $P_{-1}$ be the last vertex in~$P$ and $P_{-2}$ the second-to-last.
Given NBWs $P$ and $P'$ with $P_{-2} P_{-1} = P'_{0} P'_1$, the concatenation $P \circ P'$ first follows $P$ to $P_{-1}$ and then follows $P'$ from $P'_1$ until its end.
By $P_{\leq i}$ we denote the prefix of~$P$ with $i$~edges.

Given an NBW $P=v_0\ldots v_m$  and an $i \in [m]$,
define
\begin{align*} L(i,P) &:= 
\big\{ \rho : \mbox{ 
$\rho$ is
a leaf of $T_{v_{i-1},v_i}$, and 
for all prefixes $P'$ of $\ell(\rho)$, and all }\\
& \hspace{2truecm}\mbox{node labels $P''$ in $T_{P'_{-2},P'_{-1}}$ we have
$P_{\leq i} \circ P' \circ P''\notin \calI$}\,\big\},\\
L'(i,P) &:= 
\big\{ \rho : \mbox{ 
$\rho$ is
a node in $T_{v_{i-1},v_i}$ with $d/ \nn \leq \pi(\rho) < (\deg(\ell(\rho)_{-1})-1) d /\nn$}\\
& \hspace{2truecm}
\mbox{that is an ancestor of at least one leaf 
of $T_{v_{i-1},v_i}$  
in $L(i,P)$}\,\big\}.
\end{align*}
Let $F(P)  := \{i\in  [m]: \pi(L(i,P)) \geq 1/2\}$ where $\pi(L(i,P))=\sum_{\rho\in L(i,P)}\pi(\rho)$.
Note that whether or not $i\in F(P)$ depends only on $P_{\leq i}$, not on $v_{i+1},\ldots,v_m$.
``$F$'' stands for ``forward'' because with probability at least~$1/2$ the random walk passing through $T_{v_{i-1},v_i}$ cannot self-intersect after only one more tree.
Note that $\pi(L'(i,P)) \geq \pi(L(i,P))$, since every  
node  
in $L(i,P)$ has an ancestor in $L'(i,P)$.

For a node $\rho\in T_{v_{i-1},v_i}$, we write
$\pi_i(\rho)$ for its mass when it is useful to indicate the
underlying tree.

Given an NBW $P=v_0\cdots v_m$ and vertex $w\in V(G)$, define
\[
\mathcal A_w(P)
:=
\left\{
(i,\rho):
i\in F(P),\
\rho\in L'(i,P),\
\ell(\rho)_{-1}=w
\right\}.
\]

\begin{lemma}\label{lem:endpoint-mult}
If $P=v_0\cdots v_m$ is an NBW, and $w\in V(G)$, then
\[
|\mathcal A_w(P)|
\le
\lceil\log_2 d\rceil.
\]
\end{lemma}

\begin{proof}
 Suppose that
$i\in F(P)$ and $\rho\in L'(i,P)$. By the definition of
$L'(i,P)$, there is a leaf $\tilde \rho\in L(i,P)$ that is a
descendant of $\rho$. Consequently, for every prefix $Q$ of
$\ell(\rho)$, and every node
$\eta\in T_{Q_{-2},Q_{-1}}$, we have
\[
P_{\le i}\circ Q\circ\ell(\eta)\notin\mathcal I.
\tag{1}\label{eq:safe-descendant}
\]
Indeed, $Q$ is also a prefix of $\ell(\tilde\rho)$, so this follows
directly from the definition of $L(i,P)$.

Consider now
\[
(i,\rho),(j,\rho')\in\mathcal A_w(P), \quad i\leq j\, .
\]
We will show that $\ell(\rho)$ is a suffix of $\ell(\rho')$ or vice versa. 
Suppose, for a contradiction, that neither $\ell(\rho)$ nor
$\ell(\rho')$ is a suffix of the other. Since both walks end at $w$,
they have a nonempty maximal common suffix. Let $u$ be the first
vertex of this maximal common suffix. By assumption, the predecessors of $u$ in
$\ell(\rho)$ and $\ell(\rho')$ are distinct; denote them by $x$ and
$y$ respectively.

Let
\[
A=a_0a_1\cdots a_r
\]
be the prefix of $\ell(\rho)$ ending at $a_r=u$, so that
$a_{r-1}=x$, and let $Q$ be the prefix of $\ell(\rho')$ ending at
$u$. Define
\[
R:=yua_{r-1}a_{r-2}\cdots a_0.
\]
Because $x\neq y$ and the reverse of an NBW is an NBW, $R$ is an
NBW beginning with the last directed edge $yu$ of $Q$.

The mass of $R$ in $T_{y,u}$ satisfies
\[
\pi(R)
=
\frac{\pi_i(A)}{\deg(u)-1}
\ge
\frac{\pi_i(\rho)}{d}
\ge
\frac{1}{\sqrt n}\, .
\]
Here we used that $A$ is a prefix of $\ell(\rho)$, that
$\deg(u)-1\le d$, and that
$\pi_i(\rho)\ge d/\sqrt n$. It follows that $R$ is the label of
a node of $T_{y,u}$.

On the other hand,
\[
P_{\le j}\circ Q\circ R\in\mathcal I,
\]
because $R$ returns to $a_0=v_{i-1}$, which has already appeared
in $P_{\le j}$. This contradicts
\eqref{eq:safe-descendant}, applied to $\rho'$ and the prefix
$Q$. Thus $\ell(\rho)$ is a suffix of $\ell(\rho')$ or vice versa as claimed. 

We also note that if 
\(
(i,\rho),(j,\rho')\in\mathcal A_w(P)
\)
are distinct then the labels $\ell(\rho), \ell(\rho')$ are
distinct.
Indeed, since $j\in F(P)$, $P_{\le j}$ is
self-avoiding: otherwise $L(j,P)=\emptyset$, contradicting
$j\in F(P)$. Hence, if
$\ell(\rho)=\ell(\rho')$, equality of their first directed edges
implies $i=j$ ($v_i=v_j$ implies $i=j$ since $P_{\leq j}$ is self-avoiding). Since nodes of a fixed tree are uniquely determined
by their labels, this then implies $\rho=\rho'$.

Now choose $(i_0,\rho_0)\in\mathcal A_w(P)$ so that
$\ell(\rho_0)$ has minimum length. By the above, we have that $\ell(\rho_0)$ is a suffix of every other label
in $\mathcal A_w(P)$.

For $(i,\rho)\in\mathcal A_w(P)$, let $q=q(i,\rho)$ be the
number of edges of $\ell(\rho)$ preceding 
$\ell(\rho_0)$. Every transition probability along these $q$ edges
is at most $1/2$, because the minimum degree of $G$ is at least $3$.
Therefore
\[
\pi_i(\rho)
\le
2^{-q}\pi_{i_0}(\rho_0).
\]
Using the defining mass bounds for $L'(i,P)$ and
$L'(i_0,P)$, we obtain
\[
\frac{d}{\sqrt n}
\le
\pi_i(\rho)
\le
2^{-q}\pi_{i_0}(\rho_0)
<
2^{-q}\bigl(\deg(w)-1\bigr)\frac{d}{\sqrt n}.
\]
It follows that
\[
0\le q<
\log_2\bigl(\deg(w)-1\bigr).
\]
The result follows by recalling that each element of $\mathcal{A}_w(P)$ has a unique label and distinct labels give distinct values of $q$ (since all labels are suffix-comparable). The result follows. 
\end{proof}

\begin{lemma}\label{lem:disjoint} 
There is a constant $C=C(d)$ such that, 
for any NBW $P$, $|F(P)| \leq  C \sqrt{n}$.
\end{lemma}

\begin{proof}
For every $i\in F(P)$, we have
\(
\pi(L'(i,P))
\ge
1/2.
\)
Every $\rho\in L'(i,P)$ satisfies
\(
\pi(\rho)
\leq
d^2/\sqrt n
\)
and so 
\(
| L'(i,P) |
\ge
\sqrt n/(2d^2)\, .
\)
Summing over $i\in F(P)$ and then grouping the resulting pairs
$(i,\rho)$ according to the final vertex of $\ell(\rho)$, we
obtain
\[
\lvert F(P)\rvert\frac{\sqrt n}{2d^2}
\le
\sum_{i\in F(P)}\lvert L'(i,P)\rvert\\
=
\sum_{w\in V(G)}\lvert\mathcal A_w(P)\rvert\\
\le
n\lceil\log_2 d\rceil,
\]
where the final inequality follows from
Lemma~\ref{lem:endpoint-mult}. The result follows.
\end{proof}

We can now prove our first main result, which we re-state for convenience.

\thmUB*
\begin{proof}
Let $v_0=u$ and $v_1\in N(u)$ be chosen uniformly at random.
Consider an NBW starting from the one-edge walk $P[1] = v_0 v_1$.   
The NBW with this start 
will be decomposed as $P[1] \circ P[2] \circ \cdots$, 
where for each $j>1$ a leaf $\rho_j$ of $T_{P[j-1]_{-2},P[j-1]_{-1}}$ is chosen with 
probability $\pi(\rho_j)$ and $P[j] := \ell(\rho_j)$. 

 Given any odd number $j$ such that $P[1] \circ \cdots \circ P[j]\notin \calI$,
we wish to upper-bound the probability that
$P[1] \circ \cdots \circ P[j+2] \notin \calI$. 
Let
$
P^{\circ j}:=P[1]\circ\cdots\circ P[j]
$
and 
$i_j:=|E(P^{\circ j})|$.
Thus $P^{\circ j}=P_{\le i_j}$, and its last vertex is
$P^{\circ j}_{-1}=v_{i_j}$.
If $i_j\in F(P^{\circ j})$ then we upper-bound the probability that $P^{\circ j+2} \notin \calI$ trivially by~$1$.
 We next claim:
\begin{equation}\label{eq:claim prob}
\text{If } 
i_j\notin F(P^{\circ j}),
\text{ then } \Pr[P^{\circ j+2}\notin \calI \mid P^{\circ j}\notin \calI]\leq 1-\frac{1}{2d\sqrt{n}}.
\end{equation}

Indeed, we have by definition $\pi(L(i_j,P^{\circ j}))< 1/2$.
Therefore $\ell(\rho_{j+1})\notin L(i_j,P^{\circ j})$ with probability at least~$1/2$.
So, with probability at least $1/2$, the walk goes to a prefix $P'$  
of a path corresponding to a leaf in $T_{P^{\circ j}_{-2},P^{\circ j}_{-1}}$ such that there exists a node label $P'' \in T_{P'_{-2},P'_{-1}}$ such that $P^{\circ j} \circ P'\circ P''$ self-intersects.  
The probability of following $P''$ is 
at least $1/(d \sqrt{n})$. 
So the probability that 
$P[1] \circ \cdots \circ P[j+2] \notin \calI$ is at most $1-(1/2) (1/(d \sqrt{n}))$, proving~\eqref{eq:claim prob}.

Now take any odd $j>1$. Since $\{ P^{\circ j} \notin \calI \} \subseteq \{ P^{\circ (j-2)} \notin \calI \}$, and $P^{\circ 1} \notin \calI$, we have a telescoping product 
\begin{align*}
\Pr[P^{\circ j }\notin \calI]=\prod_{r=1}^{(j-1)/2} \Pr[P^{\circ(1+2r)}\notin \calI\mid P^{\circ(1+2r-2)}\notin \calI ].
\end{align*}

Let $k =  \lfloor C\sqrt{n} \rfloor $, with $C$ the constant from Lemma~\ref{lem:disjoint}. 
Let $P=P[1]\circ P[2]\circ \cdots$.
Since at most $k$ of the numbers $1+2r$ satisfy $i_{1+2r}\in F(P)$ by Lemma~\ref{lem:disjoint}, we obtain that 
\[
\Pr[P^{\circ j}\notin \calI]\leq \left(1-\frac{1}{2d\nn}\right)^{(j-1)/2-k}.
\]
We will upper bound this with $1$ if $(j-1)/2 \leq k$.
 
Returning to the expected collision time, denote by $S(P)$ the first time that the NBW $P=P[1]\circ P[2]\circ \cdots$ self-intersects. This is a positive integer valued random variable, hence 
$$
\mathbb{E}[S(P)]=\sum_{i\geq 1} \Pr[S(P)\geq i] = \sum_{i\ge 1} \Pr[P_{\leq i-1} \notin \calI].
$$
Now notice that  the probabilities $\Pr[P_{\leq i-1} \notin \calI]$ are  non-increasing with $i$. Recall that we denote $P^{\circ j}_{-1}$ by $v_{i_j}$. 
Since the minimum degree of~$G$ is at least $3=2+1$,  there are at most $O(\log n)$ steps from $v_{i_j}$ to $v_{i_{j+2}}$ and we can therefore upper-bound all the non self-intersecting probabilities of $P_{\leq i}$ for $i_j\leq i\leq i_{j+2}$ by $\Pr[P^{\circ j} \notin \calI]$.
Therefore we obtain
\begin{align*}
\mathbb{E}(T(G,u))=&\sum_{i\geq 1}\Pr[P_{\leq i-1} \notin \calI]\\
\leq &O(\log n)\sum_{\substack{j\geq 1\\ \text{$j$ odd}}} \Pr[ P^{\circ j} \notin \calI]
\\
\leq &O(\log n)\left (O(\sqrt{n})+\sum_{\substack{j\geq 2k+1 \\ \text{ $j$ odd}}} \left(1-\frac{1}{2d \sqrt{n}}\right)^{(j-1)/2-k}\right)
\\
=&O(\sqrt{n}\log n),
\end{align*}
as desired.
\end{proof}

\section{Lower Bound on self-intersection time}
 
In this section we prove Theorem~\ref{thm:LB}, which we re-state for convenience.
\thmLB*

\begin{proof} 
A construction of Lubotzky, Phillips and Sarnak~\cite{LPS} gives an increasing sequence of regular bipartite graphs of degree $\Delta=p+1$ such that each graph $G$ in the sequence has the following properties:  
\begin{description}
\item[\rm(LPS1)] The absolute value of every non-trivial eigenvalue of the adjacency matrix of~$G$ is $\leq2\sqrt p$.  (A~graph with this property is sometimes called a ``Ramanujan graph''.)  
Note that the trivial eigenvalues of $G$ are $\pm(p+1)$, and $2 \sqrt p < p+1$.

\item[\rm(LPS2)] The girth of $G$ is at least $2g+1$ where $g=(\frac23-o(1))\log_p n$ and $n=|V(G)|$ is the order of $G$.  For definiteness, let's say that $g=\lceil\frac35\log_p n\rceil$,  The important point for us is that $g=(\frac12+\Omega(1))\log_p n$.
\end{description}

Suppose $(X_t:t\geq 0)$ is an NBW on~$G$.  For  $a\geq0$ and $s\geq1$, let $Y_a^s$ be the indicator function for the event $X_{a+s}=X_a$.  We will estimate the expectation of~$Y_a^s$ depending on the size of $s$. 

\begin{claim}
For $a\geq 0$,
\[
\mathbb{E}\left[Y_a^s\right]
\le
\begin{cases}
0,
& \text{if } s\le 2g,\\
n^{-3/5},
& \text{if } s>2g,\\
3/n,
& \text{if } s> 4\log_p n.
\end{cases}
\]
\end{claim}
\begin{proof}
If $s\leq 2g$ then $\Exp(Y_a^s)=0$, since $G$ has girth greater than $2g$.

Suppose now that $s>2g$ and let $v=X_a$.  Consider the ball $B_g(v)$ of radius~$g$ centred at vertex~$v$.  By (LPS2), the induced subgraph $G[B_g(v)]$ is a $\Delta$-regular tree $\mathbb T$ of depth~$g$.  The only way that the event $X_{a+s}=v$ can occur is if $X_{a+s-g}$ is a leaf of $\mathbb T$ and the next $g$~steps of the NBW are all towards~$v$.  (If at any time-step the walk moves away from $v$ it must continue moving away from~$v$ until it again reaches a leaf.)  The probability that the walk makes $g$ consecutive steps towards~$v$ is $p^{-g}\leq n^{-3/5}$.

Finally we consider the case
$
s\geq 
\tau:=\left\lceil4\log_p n\right\rceil.
$
Recall the Markov chain $\vec P_G$ on the state space $\vec E(G)$ of directed edges of $G$ from the introduction. We use a mixing estimate for this chain on bipartite Ramanujan graphs $G$ from
\cite{LubetzkyPeres}. Let
$
N:=|\vec E(G)|=n(p+1).
$
For a directed edge $e$, let $\mu_t^e$ denote the law after $t$
steps of $\vec P_G$ starting from $e$,
and let $\pi_{e,t}$ denote the uniform distribution on the $N/2$
directed edges in the parity class reachable from $e$ at time $t$.
The final estimate in the proof of \cite[Corollary~3.8]{LubetzkyPeres} gives
\begin{align}\label{eq:LP-Bipartite}
\left\|
\frac{\mu_t^e}{\pi_{e,t}}-1
\right\|_{L^2(\pi_{e,t})}^2
\le
Np^{-t}\left(4pt^2+1\right)
\end{align}
for every directed edge $e$ and every $t\ge1$.

For all sufficiently large $n$, Cauchy--Schwarz and the estimate~\eqref{eq:LP-Bipartite} imply that
\[
\|\mu_t^e - \pi_{e,t}\|_{\textrm{TV}}
\leq
\left\|
\frac{\mu_t^e}{\pi_{e,t}}-1
\right\|_{L^2(\pi_{e,t})}
\le
\frac1n
\]
for every $e\in\vec E(G)$ and every $t\ge\tau-1$.

For $v\in V(G)$, let
\[
F_v:=\{(w,v):w\sim v\}
\]
be the set of directed edges ending at $v$. Whenever $F_v$ belongs
to the parity class supporting $\pi_{e,t}$, we have
\[
\pi_{e,t}(F_v)
=
\frac{|F_v|}{N/2}
=
\frac{p+1}{n(p+1)/2}
=
\frac2n.
\]
We thus have
\[
\mu_t^e(F_v)
\le
\pi_{e,t}(F_v)
+
\|\mu_t^e - \pi_{e,t}\|_{\textrm{TV}}
\le
\frac3n.
\]
If $F_v$ is not in the reachable parity class, then $\mu_t^e(F_v)=0$.

It follows that, for every $a\ge0$ and $s\ge\tau$, conditioning on
the directed edge $(X_{a},X_{a+1})$
gives
\[
\Pr(X_{a+s}=X_a\mid X_{a},X_{a+1})
\le
\frac3n
\]
and so
\[
\mathbb{E}\left[Y_a^s\right]=
\Pr(X_{a+s}=X_a)\le\frac3n,
\]
finishing the proof of the claim.
\end{proof}

To conclude, we note that by linearity of expectation, the expected number of self-intersections of the random walk by any time $t_1$ is given by
\begin{align*}
\sum_{a=0}^{t_1-1}\sum_{s=1}^{t_1-a}\Exp(Y_a^s)
&\leq t_1\Big[ 2g\times0+4\log_pn\times n^{-3/5}+t_1\times \frac3n\Big]\\
&=\frac{3t_1^2}n+4\,t_1n^{-3/5}\log_pn.
\end{align*}
Setting $t_1=\frac13\sqrt n$, the last expression is bounded above by $\frac12$ for sufficiently large $n$.  Thus, the expected time to the first self intersection is at least $\frac16\sqrt n$. (To see this, let~$T$ be the first self intersection time. Then
$\Pr(T \leq t_1)$ is at most the expected number of  self-intersections up to time~$t_1$ which is at most~$1/2$. So $\Exp(T) = \sum_{t\geq0} \Pr(T>t) \geq \sum_{t=0}^{t_1-1} \Pr(T>t) \geq\sum_{t=0}^{t_1-1} \Pr(T>t_1) \geq t_1/2$.) 
\end{proof}

\section{An improved upper bound for expanders}

In this section, we show that in the case of regular graphs with a uniform spectral gap, we can improve the bound on the expected self-intersection time of non-backtracking random walks to $O(\sqrt{n})$. The previous section shows that this bound is best possible for this class of graphs. The result is a simple consequence of the birthday paradox for Markov chains due to Kim, Montenegro, Peres, and Tetali~\cite[Theorem 3.6]{kim2008birthday} along with results of Lubetzky and Peres on the mixing time of non-backtracking random walks~\cite{LubetzkyPeres}.

Recall that given a $\Delta$-regular graph $G$, we let
\[
\lambda(G):=\max \{|\lambda|: \lambda\in \textrm{Spec}(A_G), \lambda\notin\{\Delta, -\Delta\}\}
\]
denote the largest absolute, nontrivial eigenvalue of $G$. We will assume that $\lambda(G) \leq \Delta-\eps$ for some uniform $\eps>0$.

\begin{comment}
We first note that the non-backtracking walk is not a Markov chain when viewed only on vertices, since the law of the next step depends on the previous vertex. However, as observed by... it is a Markov chain on directed edges. 
Indeed, let
\[
\vec E(G):=\{(x,y):\{x,y\}\in E(G)\}
\]
be the set of directed edges, so that
\(
N:=|\vec E(G)|=nd.
\)
For \(t\ge 1\), define
\[
Y_t:=(X_{t-1},X_t).
\]
Then \((Y_t)_{t\ge 1}\) is a Markov chain on \(\vec E(G)\), with transition
matrix
\[
\vec P_G((x,y),(y,z))
=
\begin{cases}
\dfrac1{d-1}, & z\sim y,\ z\neq x,\\[4pt]
0, & \text{otherwise}.
\end{cases}
\]
Since \(G\) is \(d\)-regular, \(\vec P_G\) is doubly stochastic, hence its stationary
distribution is uniform on \(\vec E(G)\). Moreover, since \(G\) is connected,
non-bipartite, and \(d\ge 3\), this directed-edge chain is ergodic.
\end{comment}

We again use the Markov chain $\vec P_G$ on the state space $\vec E(G)$ of directed edges from the introduction. Recall that $N:=|\vec E(G)|=n\Delta$.
Before turning to the proof of Theorem~\ref{thm:UBexpander}, we need two preliminary lemmas establishing rapid mixing for $\vec P_G$ treating the cases where $G$ is bipartite/non-bipartite separately. We note that \(\vec P_G\) is not normal in general and so it is not sufficient to simply bound the eigenvalues of \(\vec P_G\). Instead we use a spectral decomposition of  $\vec P_G$ due to Lubetzky and Peres~\cite{LubetzkyPeres}.

\begin{lemma}\label{lem-LP-mix}
Fix \(\Delta\ge 3\) and \(\varepsilon>0\). There exist constants
\(\rho_0=\rho_0(\Delta,\epsilon)<1\) and \(C_0=C_0(\Delta,\varepsilon)\) such that the following holds. Let \(G\) be a connected non-bipartite \(\Delta\)-regular graph on \(n\) vertices. If
\[
\lambda(G)\le \Delta-\varepsilon\, ,
\] 
then for all $t\geq 1$ and all $e,f\in \vec E(G)$
\[
\left|\vec P_G^t(e,f)-\frac1N\right|
\le
C_0\rho_0^t\,.
\]
\end{lemma}

\begin{proof}
Lubetzky and Peres~\cite[Proposition 3.1]{LubetzkyPeres} showed that $\vec P_G$ is unitarily similar to a block-diagonal matrix with $1\times1$ and $2\times2$ blocks on the diagonal.
More precisely, they show there is a unitary matrix \(U\)
such that 
\[
U^\ast \vec P_G U
=\frac{1}{\Delta-1}\cdot 
\operatorname{diag}\left(
\Delta-1,\,
\begin{pmatrix}
\theta_2 & \alpha_2\\
0 & \theta_2'
\end{pmatrix},
\,\ldots,\,
\begin{pmatrix}
\theta_n & \alpha_n\\
0 & \theta_n'
\end{pmatrix},
\,-1,\ldots,-1,\,
1,\ldots,1
\right)
\]
where \(|\alpha_i|<2(\Delta-1)\) and
\(\theta_i,\theta_i'\) are the roots of
\[
\theta^2-\lambda_i\theta+\Delta-1=0,
\]
where $\lambda_1\geq\ldots\geq\lambda_n$
are the eigenvalues of $G$. The leading one-dimensional block corresponds to the constant
vector.

By our assumption that
\(
\lambda(G)\le \Delta-\varepsilon,
\)
there is a constant
\(r=r(\Delta,\varepsilon)<1\) such that
\[
\frac{|\theta_i|}{\Delta-1}\le r,
\qquad
\frac{|\theta_i'|}{\Delta-1}\le r
\]
for all $i\in \{2,\ldots,n\}$. 
Let 
\(
M_i=(\Delta-1)^{-1}\begin{pmatrix}
\theta_i & \alpha_i\\
0 & \theta_i'
\end{pmatrix}
\). For \(t\ge1\),
\[
M_i^t
=
\frac1{(\Delta-1)^t}
\begin{pmatrix}
\theta_i^t &
\alpha_i\sum_{k=0}^{t-1}
\theta_i^{\,t-1-k}(\theta_i')^k\\
0 & (\theta_i')^t
\end{pmatrix}.
\]
The diagonal entries satisfy
\[
\left|\frac{\theta_i}{\Delta-1}\right|^t\le r^t,
\qquad
\left|\frac{\theta_i'}{\Delta-1}\right|^t\le r^t.
\]
For the upper-right entry,
\begin{align*}
\left|
\frac{\alpha_i}{(\Delta-1)^t}
\sum_{k=0}^{t-1}
\theta_i^{\,t-1-k}(\theta_i')^k
\right|
\le
\frac{|\alpha_i|}{(\Delta-1)^t}
\sum_{k=0}^{t-1}
|\theta_i|^{t-1-k}|\theta_i'|^k
\le
\frac{|\alpha_i|}{(\Delta-1)^t}\,t((\Delta-1)r)^{t-1}
\le
2t r^{t-1}.
\end{align*}

Bounding the operator norm by the Frobenius norm (the $L^2$-norm of the vector of all the entries) we therefore have
\[
\|M_i^t\|_{2\to2}
\le
\left(2r^{2t}+4t^2r^{2t-2}\right)^{1/2}
\le
\sqrt2\,r^t+2t r^{t-1}.
\]
In particular, choosing \(\rho_0\in (r,1)\) we have 
\[
\|M_i^t\|_{2\to2}
\le C_0\rho_0^t
\]
for some constant \(C_0=C_0(\Delta,\epsilon)\).

Now let \(\Pi\) denote the orthogonal projection onto the constant vector in
\(\mathbb C^{\vec E}\).
Since the stationary distribution of the directed-edge chain is uniform, we have
\begin{align}\label{eq:P-mix-bound}
\left|\vec P_G^t(e,f)-\frac1N\right|\leq\|\vec{P}_G^t-\Pi\|_{2\to2}
=
\max\left\{
\max_i\|M_i^t\|_{2\to2},
(\Delta-1)^{-t}
\right\}
\le
C_0\rho_0^t
\end{align}
for every \(e,f\in\vec E\)
after increasing
\(\rho_0\) if necessary so that \(\rho_0>(\Delta-1)^{-1}\).
\end{proof}

We now turn to the bipartite case. 

\begin{lemma}\label{lem-LP-mix-bipartite}
Fix $\Delta\ge3$ and $\varepsilon>0$. There exist constants
$
\rho_0=\rho_0(\Delta,\varepsilon)<1
$ 
and
$
C_0=C_0(\Delta,\varepsilon)
$
such that the following holds. Let $G$ be a connected bipartite
$\Delta$-regular graph on $n$ vertices with
\[
\lambda(G)
\le
\Delta-\varepsilon.
\]
For a directed edge $e\in\vec E(G)$, let $\pi_{e,t}$ denote the uniform distribution on the
$N/2$ directed edges in the parity class reachable from $e$ at
time $t$. Then, for every $t\ge1$ and all
$e,f\in\vec E(G)$,
\[
\left|
\vec P_G^t(e,f)-\pi_{e,t}(f)
\right|
\le
C_0\rho_0^t.
\]
\end{lemma}

\begin{proof}
We follow the proof of Lemma~\ref{lem-LP-mix}, with the only
modification being that the bipartite directed-edge chain has an
additional non-decaying eigenvalue $-1$.

Fix a bipartition
$
V(G)=V_0\sqcup V_1.
$
Since $G$ is regular,
 each of the two parity classes of directed edges has
size $N/2$.

Define $\eta\in\mathbb R^{\vec E(G)}$ by
\[
\eta((x,y))
:=
\begin{cases}
1, & x\in V_0,\\
-1, & x\in V_1.
\end{cases}
\]
In the bipartite case, the unitary block decomposition of Lubetzky and
Peres~\cite[Proposition~3.1]{LubetzkyPeres} takes the form
\[
U^\ast\vec P_GU
=
\frac1{\Delta-1}
\operatorname{diag}\left(
\Delta-1,\,
-(\Delta-1),\,
\begin{pmatrix}
\theta_2&\alpha_2\\
0&\theta_2'
\end{pmatrix},
\,\ldots,\,
\begin{pmatrix}
\theta_{n-1}&\alpha_{n-1}\\
0&\theta_{n-1}'
\end{pmatrix},
\,-1,\ldots,-1,\,
1,\ldots,1
\right),
\]
where $|\alpha_i|<2(\Delta-1)$
and $\theta_i,\theta_i'$ are the roots of
$
\theta^2-\lambda_i\theta+\Delta-1=0.
$
The first one-dimensional block corresponds to the constant vector
$\mathbf 1$, while the second corresponds to $\eta$. 

The assumption $\lambda(G)\le\Delta-\varepsilon$ allows us to apply the same calculation as in the proof of
Lemma~\ref{lem-LP-mix}. As before, letting 
\[
M_i
:=
\frac1{\Delta-1}
\begin{pmatrix}
\theta_i&\alpha_i\\
0&\theta_i'
\end{pmatrix},
\]
there are constants $C_0, \rho_0$ such that 
$
\|M_i^t\|_{2\to2}
\le
C_0\rho_0^t
$
for all $i\in\{2,\ldots,n-1\}$ and all $t\ge1$. The remaining
one-dimensional blocks, whose eigenvalues have modulus
$(\Delta-1)^{-1}$, satisfy the same bound.

Let
\[
\Pi_+
:=
\frac1N\mathbf 1\mathbf 1^{\mathsf T}
\qquad\text{and}\qquad
\Pi_-
:=
\frac1N\eta\eta^{\mathsf T}
\]
be the orthogonal projections onto the eigenspaces corresponding to
the eigenvalues $1$ and $-1$, respectively. The preceding block
bounds give
\begin{equation}\label{eq:P-bipartite-mix-bound}
\left\|
\vec P_G^t-\Pi_+-(-1)^t\Pi_-
\right\|_{2\to2}
\le
C_0\rho_0^t.
\end{equation}

For $e,f\in\vec E(G)$,
\[
\begin{aligned}
\left(\Pi_++(-1)^t\Pi_-\right)(e,f)
&=
\frac{1+(-1)^t\eta(e)\eta(f)}{N}.
\end{aligned}
\]
This quantity is $2/N$ when $f$ lies in the parity class reachable
from $e$ at time $t$, and it is zero otherwise. In other words, 
$
\left(\Pi_++(-1)^t\Pi_-\right)(e,f)
=
\pi_{e,t}(f).
$
It now follows from \eqref{eq:P-bipartite-mix-bound} that
\[
\left|
\vec P_G^t(e,f)-\pi_{e,t}(f)
\right|
\le
\left\|
\vec P_G^t-\Pi_+-(-1)^t\Pi_-
\right\|_{2\to2}
\le
C_0\rho_0^t,
\]
as desired.
\end{proof}

We now restate and prove the main result of this section. 

\thmUBexpander* 

\begin{proof}
Let $(X_t)_{t\ge 0}$ be the NBW started
from $X_0=u$. Consider the corresponding directed-edge chain $(Y_t)_{t\ge 1}$ where $Y_t=(X_{t-1},X_t)$. 

We deal with the cases where $G$ is non-biparite/bipartite separately, starting with the case where $G$ is non-bipartite. 
Let
\[
\vec{T}
:=
\min\{t\ge 2:\exists\,1\le s<t\text{ with }Y_s=Y_t\}
\]
be the first self-intersection time of the directed-edge chain. 
If
\(Y_s=Y_t\), then in particular \(X_s=X_t\). Therefore
\[
T(G,u)\le \vec{T}.
\]
Thus it is enough to prove
\[
\mathbb E[\vec{T}]=O_{\Delta,\varepsilon}(\sqrt n).
\]

Let $C_0,\rho_0$ be as in Lemma~\ref{lem-LP-mix} and  
choose
\(
t_0:=\lceil K\log N\rceil
\)
with \(K=K(\Delta,\varepsilon)\) sufficiently large that
\[
C_0\rho_0^{t_0}\le \frac1{2N}.
\]
Then, for every pair \(e,f\in\vec E\),
\[
\frac1{2N}
\le
\vec P_G^{t_0}(e,f)
\le
\frac{3}{2N}.                   
\]

The birthday paradox for Markov chains \cite[Theorem 3.6]{kim2008birthday} now implies that there exist absolute constants $C_1,\delta>0$ such that 
\[
\vec{T}\leq
C_1\left(
\sqrt{
\left(
1+\sum_{j=1}^{2t_0}3j\max_{e,f} \vec{P}_G^j(e,f)
\right)N
}
+t_0
\right)
\]
with probability bounded below by an absolute constant  $\delta>0$.

By~\eqref{eq:P-mix-bound},
\[
\max_{e,f\in\vec E}\vec P_G^j(e,f)
\le
\frac1N+C_0\rho_0^j.
\]
Therefore
\[
\begin{aligned}
1+\sum_{j=1}^{2t_0}3j\max_{e,f}\vec P_G^j(e,f)
&\le
1+\frac3N\sum_{j=1}^{2t_0}j
+3C_0\sum_{j=1}^{2t_0}j\rho_0^j = O_{\Delta,\epsilon}(1) 
\end{aligned}
\]
since \(t_0=O(\log N)\) and $\rho_0<1$. 
It follows that there exists $L=O_{\Delta,\epsilon}(\sqrt{N})$
such that for every
starting directed edge \(e\),
\begin{align}\label{eq:tau-L-bound}
\mathbb P_e(\vec{T}\le L)\ge \delta.              
\end{align}

Finally, we convert this probability bound into an expectation bound. We note that by~\eqref{eq:tau-L-bound} and the Markov property we have
\[
\mathbb P(\vec{T}>(k+1)L\mid \vec{T}>kL)\le 1-\delta.
\]
Iterating gives
\[
\mathbb P(\vec{T}>kL)\le (1-\delta)^k.
\]
Therefore
\[
\mathbb E[\vec{T}]
=
\sum_{t\ge 0}\mathbb P(\vec{T}>t)
\le
L\sum_{k\ge 0}\mathbb P(\vec{T}>kL)
\le
L\sum_{k\ge 0}(1-\delta)^k
=
\frac{L}{\delta}
=
O_{\Delta,\varepsilon}(\sqrt n).
\]
Since \(T(G,u)\le \vec{T}\), this completes the proof in the case where $G$ is non-bipartite.

Suppose now that $G$ is bipartite. Write
$
V(G)=V_0\sqcup V_1
$
and let
\[
\vec E_i:=\{(x,y)\in\vec E(G):x\in V_i\}.
\]
Consider the two-step chain
\[
Q_i:=\vec P_G^2\big|_{\vec E_i}\, .
\]
By Lemma~\ref{lem-LP-mix-bipartite}, for every $e,f\in\vec E_i$,
\[
\left|
Q_i^t(e,f)-\frac2N
\right|
\le
C_0\rho_0^{2t}.
\]
Suppose wlog that $u\in V_0$ and set
\[
Z_k:=Y_{2k+1},
\qquad k\ge0,
\]
and let
\[
S:=\min\{k\ge1:\exists\,0\le\ell<k\text{ with }Z_\ell=Z_k\}.
\]
Since $Z_0=Y_1\in\vec E_0$, the birthday-paradox argument above,
applied to $Q_0$, gives
\[
\mathbb E[S]
=
O_{\Delta,\varepsilon}(\sqrt n).
\]
A self-intersection $Z_\ell=Z_k$ is also a self-intersection of the
full directed-edge chain, and hence
$
\vec T\le2S+1.
$
Therefore,
\[
\mathbb E[T(G,u)]
\le
\mathbb E[\vec T]
\le
2\mathbb E[S]+1
=
O_{\Delta,\varepsilon}(\sqrt n),
\]
which completes the bipartite case.

\end{proof}

\section{An application}\label{sec:appli}
The self-intersection time of the non-backtracking random walk (NBW) gives an upper bound on the spectral independence (SI) of the Ising model at critical temperature. 
Consequently, our bounds on the NBW self-intersection time imply improved SI bounds, and hence improved upper bounds on the mixing time of Glauber dynamics for critical Ising models. 

We introduce some definitions before stating our application.
Let $\mu$ be a distribution over $\{-1,+1\}^V$. 
Let $X \sim \mu$. The \emph{influence matrix} $\Psi_{\mu}\in\mathbb{R}^{V\times V}$ is defined as: 
\begin{align*}
\Psi_{\mu}(u,v)=\begin{cases}
\Pr(X_v=+1\mid X_u=+1)-\Pr(X_v=+1\mid X_u=-1) &\quad \text{if }\mathbb{E}(X_u)\in(-1,+1)\\
&\qquad \text{and }u\ne v;\\
0 &\quad \text{otherwise}.
\end{cases}
\end{align*}
For any vertex $u \in V$, the \emph{total influence} of $u$ in $\mu$ is defined as
\begin{align*}
\mathrm{TI}^{\mu}_{u} = 1+\sum_{v\in V}|\Psi_{\mu}(u,v)|.
\end{align*}

The total influence serves as an upper bound for the spectral independence (SI) which we define now.
Given any subset $\Lambda \subseteq V$, a partial assignment $\tau \in \{-1,+1\}^{\Lambda}$ specified on $\Lambda$ is called a \emph{feasible pinning} if it can be extended to a full assignment $\sigma \in\{-1,+1\}^V$ such that $\sigma_\Lambda=\tau$ and $\mu(\sigma)>0$.

We say that $\mu$ is \emph{$\rho$-spectrally independent} if for every feasible pinning $\tau$ on any subset $\Lambda$,
\begin{align*}
\lambda_{\max}(\Psi_{\mu^\tau}) \le \rho,
\end{align*}
where $\mu^{\tau}$ denotes the conditional distribution induced by $\mu$ given the boundary condition $\tau$, and $\lambda_{\max}$ is the largest eigenvalue of the matrix $\Psi_{\mu^\tau}$. Since the largest eigenvalue of the influence matrix is bounded above by its maximum absolute row sum, it follows that $\mu$ is $\rho$-spectrally independent if, for every $u\in V$ and every feasible pinning $\tau$, we have $\mathrm{TI}^{\mu^\tau}_{u}\le 1+\rho$.

Let $G=(V,E)$ be a graph. Let $\beta \in \mathbb{R}$ denote the inverse temperature, and let $\boldsymbol{h} \in \mathbb{R}^V$ denote the external field. The Gibbs distribution for the Ising model is given by:
\begin{align*}
\forall \boldsymbol{x} \in \{-1,+1\}^V, \qquad \mu(\boldsymbol{x}) \propto \exp\left(\frac{\beta}{2} \boldsymbol{x}^\intercal A_G \boldsymbol{x} + \boldsymbol{h}^\intercal \boldsymbol{x}\right),
\end{align*}
where $A_G$ is the adjacency matrix of $G$. When $\beta>0$, the Ising model is called \emph{ferromagnetic}, while it is called \emph{antiferromagnetic} when $\beta<0$.

For every $\Delta\ge 3$, let $\beta_c(\Delta)>0$ denote the tree-uniqueness threshold of the zero-field Ising model on the infinite $\Delta$-regular tree, defined by
\[
(\Delta-1)\tanh\beta_c(\Delta)=1.
\]

The Glauber dynamics is the canonical Markov chain for sampling from the Gibbs distribution $\mu$.
Given a current state $X$ in the support $\Omega(\mu)$ of $\mu$, the chain updates as follows:
\begin{itemize}
\item choose a vertex $v \in V$ uniformly at random;
\item resample $X_v$ according to the marginal distribution $\mu_v(\cdot \mid X(V \setminus {v}))$ at $v$.
\end{itemize}
It is well-known that the Glauber dynamics has stationary distribution $\mu$.
Let $(Z_t)_{t \geq 0}$ denote the trajectory of the Glauber dynamics.
The \emph{mixing time} of Glauber dynamics is defined as:
\begin{align*}
    T_{\mathrm{mix}} = \max_{x \in \Omega(\mu)}\min\left\{t \mid d_{\mathrm{TV}}(\Pr(Z_t = \cdot \mid Z_0 = x),\mu) \leq 1/4\right\},
\end{align*}
where, for two distributions $\nu$ and $\mu$ on a finite set $\Omega$, the \emph{total variation distance} is defined by
    $d_{\mathrm{TV}}(\nu,\mu) = \frac{1}{2} \sum_{x \in \Omega} \left|\nu(x) - \mu(x)\right|$.

\begin{theorem}\label{thm:critical-Ising}
Fix $\Delta \geq 3$ and let $G$ be a connected $n$-vertex $\Delta$-regular graph.
Let $\mu$ be the Gibbs distribution of the Ising model on $G$ at inverse temperature $\beta$ satisfying $|\beta|=\beta_c$, where
\[
(\Delta-1)\tanh\beta_c=1.
\]
Then, for every $\Lambda\subseteq V(G)$, every feasible pinning $\tau\in\{-1,+1\}^{\Lambda}$, and every $u\in V(G)\setminus\Lambda$, the total influence of $u$ in $\mu^\tau$ satisfies
\[
\mathrm{TI}^{\mu^\tau}_u\le \frac{\Delta}{\Delta-1}\mathbb{E}(T(G,u))=O_\Delta(\sqrt{n}\log n).
\]
In particular, the Ising model is $O_\Delta(\sqrt{n}\log n)$-spectrally independent.

Moreover, if the external field satisfies $\|\boldsymbol h\|_\infty \le n^{O(1)}$, then the Glauber dynamics for $\mu$ has mixing time 
$\widetilde{O}_\Delta\!\left(n^{2+ \frac{2}{\Delta - 2}}\right)$.
\end{theorem}
\begin{proof}
Fix $\Lambda$, $\tau$, and $u$ as in the theorem. 
Let $\mathcal T=T^\tau_{\mathrm{SAW}}(G,u)$ be the self-avoiding-walk (SAW) tree rooted at $u$, equipped with the boundary condition induced by $\tau$ and by the usual cycle-closing rule. 
For $v\in V(G)$, let $\mathcal C_v$ be the set of free copies of $v$ in $\mathcal T$.

By \cite[Lemma 13(2)]{chen2020contraction}, the SAW-tree representation preserves the influences from the root:
\[
\Psi_{\mu_G^\tau}(u,v)
 =\sum_{\widehat v\in\mathcal C_v}\Psi_{\mathcal T}(u,\widehat v),
\]
where $\Psi_{\mathcal T}$ denotes the influence in the induced Ising Gibbs measure on $\mathcal T$. If $\widehat v$ is at depth $\ell$, then repeated application of the tree chain rule in \cite[Lemma 15]{chen2020contraction} expresses $\Psi_{\mathcal T}(u,\widehat v)$ as the product of the $\ell$ influences along the root-to-$\widehat v$ path. By \cite[Lemma 16]{chen2020contraction}, the absolute value of each such edge influence has the form
\[
\frac{|1-\mathrm e^{4\beta}|R}
     {(\mathrm e^{2\beta}R+1)(R+\mathrm e^{2\beta})},
\]
for some marginal ratio $R\ge 0$.
We have
\[
\alpha:=\sup_{R\in[0,\infty]}
\frac{|1-\mathrm e^{4\beta}|R}
     {(\mathrm e^{2\beta}R+1)(R+\mathrm e^{2\beta})}
 =\frac{|\mathrm e^{2\beta}-1|}{\mathrm e^{2\beta}+1}=\tanh|\beta|.
\]
It follows from the triangle inequality that
\begin{align*}
\mathrm{TI}^{\mu^\tau}_u
&\le 1+\sum_{v\in V(G)\setminus\{u\}}
       \sum_{\widehat v\in\mathcal C_v}
       |\Psi_{\mathcal T}(u,\widehat v)|\\
&\le \sum_{\ell\ge0}N_G(u,\ell)\alpha^\ell,
\end{align*}
where $N_G(u,\ell)$ is the number of self-avoiding walks of length $\ell$ in $G$ starting from $u$. Here we used that each free node of $\mathcal T$ at depth $\ell$ corresponds to a self-avoiding walk of length $\ell$ in $G$.

At $|\beta|=\beta_c$, this gives
\[
\mathrm{TI}^{\mu^\tau}_u
 \le \sum_{\ell\ge0}\frac{N_G(u,\ell)}{(\Delta-1)^\ell}.
\]
For every $\ell\ge1$, all non-backtracking walks of length $\ell$ from $u$ have probability $1/[\Delta(\Delta-1)^{\ell-1}]$, and exactly $N_G(u,\ell)$ of them are self-avoiding. Therefore
\[
\Pr(T(G,u)>\ell)
 =\frac{N_G(u,\ell)}{\Delta(\Delta-1)^{\ell-1}}
 \qquad(\ell\ge1).
\]
Using $N_G(u,0)=1$ and the identity
$\mathbb E[T(G,u)]=\sum_{\ell\ge0}\Pr(T(G,u)>\ell)$, we obtain
\begin{align*}
\mathrm{TI}^{\mu^\tau}_u
&\le 1+\frac{\Delta}{\Delta-1}
       \sum_{\ell\ge1}\Pr(T(G,u)>\ell)\\
&=\frac{\Delta}{\Delta-1}\mathbb E[T(G,u)]
  -\frac{1}{\Delta-1}\\
&\le\frac{\Delta}{\Delta-1}\mathbb E[T(G,u)].
\end{align*}
By Theorem~\ref{thm:UB}, it follows that $\mathrm{TI}^{\mu^\tau}_u=O_\Delta(\sqrt{n}\log n)$.
Notice that the argument took a supremum over the marginal ratio $R\in[0,\infty]$ and therefore holds for arbitrary external fields and any feasible pinning.

It remains to derive the mixing-time bound.
By \cite[proof of the upper bound in Theorem 1.2]{chen2025uniqueness}, the distribution $\mu$ satisfies the approximate tensorization of entropy with constant 
\[
\exp\left(2\beta_c\Delta \int_0^{1} C(t)\,\mathrm{d} t\right),
\]
where $C:[0,1] \to \mathbb{R}_{\ge 0}$ is given by
    \begin{align*}
        \forall t \in [0,1],\quad C(t) = 
            \min\left(\frac{\Delta}{\Delta-1} \cdot \frac{1}{1-(\Delta-1) \tanh (\beta_c t)}, O_\Delta(\sqrt{n}\log n)\right),
    \end{align*}
It then follows from the same calculation as in  \cite{chen2025uniqueness} that 
\begin{align*}
    2\beta_c\Delta \int_0^{1} C(t)\,\mathrm{d} t
    &\le O_\Delta(1) + 2\beta_c\Delta \int_0^{1-1/\sqrt{n}\log n} C(t)\,\mathrm{d} t\\
    &=O_\Delta(1) + 2\beta_c\Delta \int_{1/\sqrt{n}\log n}^{1} C(1-t)\,\mathrm{d} t\\
    &=O_\Delta(1) + \frac{4\beta_c \Delta}{\Delta-2} \int_{1/\sqrt{n}\log n}^{1} \frac{1}{\exp (2 \beta_c t) - 1} \,\mathrm{d} t\\
    &\le O_\Delta(1) +\left(2+\frac{4}{\Delta-2}\right)\log(\sqrt{n}\log n).
\end{align*}
Therefore, the Gibbs distribution $\mu$ satisfies the approximate tensorization of entropy with constant ${O}_{\Delta}\left((\sqrt{n}\log n)^{2+\frac{4}{\Delta-2}}\right)=\widetilde{O}_{\Delta}\left(n^{1+\frac{2}{\Delta-2}}\right)$.
Consequently, the Glauber dynamics for $\mu$ has mixing time $\widetilde{O}\left(n^{2+\frac{2}{\Delta-2}}\right)$ for fixed $\Delta\ge 3$. 
Here, we use the assumption $\|\boldsymbol h\|_\infty\le n^{O(1)}$ only to guarantee that $\log(1/\mu_{\min})=n^{O(1)}$, which is required for converting the approximate tensorization of entropy into a worst-case mixing-time upper bound.
\end{proof}

The $\widetilde{O}\!\left(n^{2+\frac{2}{\Delta-2}}\right)$ mixing-time upper bound established in Theorem~\ref{thm:critical-Ising} improves the previous state-of-the-art bound $\widetilde{O}\!\left(n^{3+\frac{4}{\Delta-2}}\right)$ for critical Ising models due to~\cite{chen2025uniqueness}.
To the best of our knowledge, the $O(\sqrt n\log n)$ upper bound on the total influence, and hence on the spectral independence, is the first sublinear spectral-independence upper bound for critical spin systems of this type. 
Moreover, this bound is tight up to the logarithmic factor, as it is matched by an $\Omega(\sqrt n)$ lower bound on the spectral independence of critical Ising models~\cite[Theorem 5.2 and Corollary 5.3]{chen2025uniqueness} .

This motivates the following broader open problem. For an $n$-vertex $\Delta$-regular graph with $\Delta\ge3$, and for a two-spin system whose parameters lie on the critical boundary of the tree-uniqueness condition, is the spectral independence always bounded by $\widetilde O(\sqrt n)$?

\section*{Acknowledgments}
Part of this work was done at the Phase Transitions in Probability Workshop at CWI in April 2026. Ferenc Bencs is supported by the Netherlands Organisation of Scientific Research (NWO): VI.Veni.222.303. Matthew Jenssen is supported by UK Research and Innovation Future Leaders Fellowship MR/W007320/2. Mark Jerrum is partly supported by grant UKRI2771: Zeros, Algorithms, and Correlation for Graph Polynomials. G\'abor Pete is supported by the Hungarian National Research, Development and Innovation Office, NKFIH Highlight grant 152849. Guus Regts is supported by the Netherlands Organisation of Scientific Research (NWO): VI.Vidi.193.068. Yitong Yin is partially supported by the New Cornerstone Science Foundation.

\bibliographystyle{plain}
\bibliography{NBRW}

\end{document}